\documentclass[12pt]{amsart}
\usepackage{latexsym,amssymb,amsmath}
\usepackage{latexsym,amssymb,amsmath,tikz,verbatim,cancel}

\numberwithin{equation}{section}
\newtheorem{theorem}{Theorem}[section]
\newtheorem{lemma}[theorem]{Lemma}

\newtheorem{corollary}[theorem]{Corollary}

\theoremstyle{definition}
\newtheorem{definition}[theorem]{Definition}

\newtheorem{remark}[theorem]{Remark}
\newtheorem{example}[theorem]{Example}

\begin{document}

\title{Virtual resolutions of points in $\mathbb{P}^1 \times \mathbb{P}^1$}

\author[M. Harada]{Megumi Harada}
\address{Department of Mathematics and Statistics\\
McMaster University, Hamilton, ON, L8S 4L8, Canada}
\email{megumi.harada@math.mcmaster.ca}

\author[M. Nowroozi]{Maryam Nowroozi}
\address{Department of Mathematics and Statistics\\
McMaster University, Hamilton, ON, L8S 4L8, Canada}
\email{nowroozm@mcmaster.ca}

 \author[A. Van Tuyl]{Adam Van Tuyl}
\address{Department of Mathematics and Statistics\\
McMaster University, Hamilton, ON, L8S 4L8, Canada}
\email{vantuyl@math.mcmaster.ca}

\date{\today}

\begin{abstract}
We  explore explicit virtual resolutions, 
as introduced by 
Berkesch, Erman, and Smith, for ideals of 
finite sets of points in $\mathbb{P}^1 \times \mathbb{P}^1$.  
Specifically, we describe a virtual resolution for a sufficiently 
general set of points $X$ in
$\mathbb{P}^1 \times \mathbb{P}^1$ that
only depends on $|X|$.  We also improve an existence 
result of Berkesch, Erman, and Smith in the special case of points
in  $\mathbb{P}^1 \times \mathbb{P}^1$;  more precisely, 
we give an effective bound for their construction 
that gives a virtual resolution of length
two for any set of points in $\mathbb{P}^1 \times \mathbb{P}^1$.
\end{abstract}

\maketitle

%%%%%%%%%%%%%%%%%%%%%%%%%%%%%%%%%%%%%%%%%%%%%%%%%%%%%%

\section{Introduction}

Minimal free resolutions of ideals in the coordinate ring of a toric variety are well-studied in commutative algebra and algebraic geometry as a method of obtaining information about the corresponding subvariety. It has been observed, however, that minimal free resolutions can be longer than necessary in order to obtain useful information. The theory of virtual resolutions was introduced by Berkesch, Erman, and Smith \cite{Virtual} as a useful analogue to the theory of minimal free resolutions which often provide \emph{shorter} resolutions. In this sense, virtual resolutions are a useful new tool to study varieties in toric varieties, including multi-projective spaces. 

The theory of virtual resolutions is quite new, so there are abundant open questions, and a lack of (families of) concrete examples.  
 On the other hand, since their introduction in \cite{Virtual}, they have been actively studied. In their original work \cite{Virtual}, Berkesch, Erman, and Smith gave multiple strategies for constructing virtual resolutions.  Loper \cite{Loper2} identified two algebraic conditions that characterize when a chain complex is virtual. Kennedy \cite{Jarvis} gave an
 equivalent algebraic definition for a complex to be
 a virtual resolution.
 The concept of being virtually Cohen-Macaulay,
 introduced in \cite{Virtual}, is
 studied in the context of simplicial
 complexes   by Kenshur, Lin, McNally, Xu, and Yu \cite{KLMXY}, and more generally
 in the recent paper of  Berkesch, Klein, Loper, and Yang \cite{Berkesch}.  Yang \cite{Yang}  also studied virtual
 resolutions of monomial ideals on toric varieties. More recently,
 Booms-Peot and Cobb \cite{Booms} have examined when a 
 generalized Eagon-Northcott complex is a virtual resolution.
 
 In this paper we focus on virtual resolutions of finite sets of points in $\mathbb{P}^1 \times
\mathbb{P}^1$.  By focusing on this family, we can take advantage of many
known algebraic and geometric properties about them, e.g., as 
developed in \cite{FGM,GMR2,GMR1,G,Adam}.  At the same time,
we complement Gao, Li, Loper, and Matto's paper \cite{Loper2} which studies points
in $\mathbb{P}^1 \times \mathbb{P}^1$ that
are virtual complete intersections.
The contribution of this manuscript 
is two-fold, as we now describe.

Our first main result gives explicit constructions of a family of virtual resolutions of length two for 
a set $X$ of points in sufficiently
general position (as defined in Section~\ref{backgroundsection}) in $\mathbb{P}^1 \times \mathbb{P}^1$. Note that  we have suppressed the virtual resolution in the statement below, but the complete details can be found in Section 3.

\begin{theorem}[Theorem \ref{maintheoremtrim}]
Let $X$ be a finite set of points in 
$\mathbb{P}^1 \times \mathbb{P}^1$ in sufficiently general
position. Assume $\lvert X \rvert > 1$ and suppose that
there exists an $(i,j) \in \mathbb{N}^2$ with $i < j$ such that
$|X| = (i+1)(j+1)$.   Then there is a virtual resolution of $S/I_X$ of length two which can be explicitly described in terms of $i$ and $j$.
\end{theorem}
\noindent
We note that, in general, the minimal free resolutions of such $X$ are known to have length three, so our result shows that we can shorten the resolution if we allow for virtual resolutions. The main tool in our construction is a ``trimming'' result of Berkesch, Erman, and Smith \cite[Theorem 1.3]{Virtual} (stated below as Theorem~\ref{theorem:trim}), which roughly states that well-chosen (graded) subcomplexes of a (graded) minimal free resolution $\mathcal{F}_{\bullet}$ of $S/I_X$ is in fact a virtual resolution of $S/I_X$. We also employ results of
Giuffrida, Maggioni, and Ragusa \cite{GMR2,GMR1} which gives concrete information about the second-difference function $\Delta^2 H_X$ (defined in Section~\ref{backgroundsection}) of the Hilbert function and its relation to the minimal free resolution $\mathcal{F}_{\bullet}$ of $S/I_X$. This information
is used to prove that for certain choices of these graded subcomplexes, the resulting virtual resolution is of length two. 

Our second main result improves on another result of Berkesch, Erman, and Smith (in the special case of points in $\mathbb{P}^1 \times \mathbb{P}^1$), in the following sense. In addition to the ``trimming'' result mentioned above, Berkesch, Erman, and Smith give another method for constructing virtual resolutions in \cite[Theorem 4.1]{Virtual}.  Let $S = \Bbbk[x_0, x_1, y_0, y_1]$ be the bihomogeneous coordinate ring of $\mathbb{P}^1 \times \mathbb{P}^1$. In the case of $\mathbb{P}^1 \times \mathbb{P}^1$, their theorem states that for some positive integer $a$, the minimal free resolution of  $S/(I_X \cap \langle x_0, x_1 \rangle^a)$ is a virtual resolution of length two
of the original ring $S/I_X$. Their result is an existence result in the sense that they do not give a concrete value for the positive integer $a$. Our contribution is to give an explicit lower bound for the value of $a$ 
for any set of points $X$ in $\mathbb{P}^1 \times \mathbb{P}^1$.  More precisely, let $\pi_1(X)$  denote the set of distinct first coordinates which appear in the points in $X$. Then we show:

\begin{theorem}[Theorem \ref{maintheorem2}]
Let $X$ be a finite set of points in $\mathbb{P}^1 \times
\mathbb{P}^1$.  Then for all $a \geq |\pi_1(X)|-1$, the 
minimal free resolution of 
$S/(I_X \cap \langle x_0,x_1 \rangle^a)$ is a virtual
resolution of $S/I_X$ of length two.
\end{theorem}

Our results suggest several other directions for future work. For instance, it would be of interest to explore whether the methods of the current manuscript can be extended to cover more general multiprojective spaces $\mathbb{P}^{n_1} \times \mathbb{P}^{n_2} \times \cdots \times \mathbb{P}^{n_s}$, possibly starting with the smallest such generalizations such as $\mathbb{P}^1 \times \mathbb{P}^2$ or $\mathbb{P}^1 \times \mathbb{P}^1 \times \mathbb{P}^1$. It would also be of interest to explore virtual resolutions of subvarieties of multiprojective spaces of dimension $>0$.

\noindent
{\bf Acknowledgments.} 
The authors thank Ayah Almousa, Daniel Erman, and Michael Loper
for useful discussions,  and the two referees for their
comments and improvements.
The authors used Macaulay2 \cite{M2},
and in particular, the package {\tt VirtualResolutions} \cite{Ayah}
in their computer experiments.  An early version of the results
in this paper appeared in the second author's MSc thesis \cite{N}.
 Harada's research is supported by NSERC Discovery Grant 2019-06567 and a Canada Research Chair (Tier 2) Award. 
 Van Tuyl’s research is supported by NSERC Discovery Grant 2019-05412.

%%%%%%%%%%%%%%%%%%%%%%%%%%%%%%%%%%%%%%%%%%%%%%%%%%%%%%%%%%%%%%%%%%%

\section{Background}
\label{backgroundsection}

In this section we recall some relevant background
on virtual resolutions and finite sets of points in $\mathbb{P}^1 \times 
\mathbb{P}^1$. Although it is not the most general setting for the theory of virtual resolutions, for the purposes of this manuscript we restrict attention to the case of the ring $S = \Bbbk[x_0, x_1, y_0, y_1]$, the $\mathbb{Z}^2$-graded coordinate ring of $\mathbb{P}^1 \times \mathbb{P}^1$, where $\Bbbk$ is a field of characteristic $0$. The $\mathbb{Z}^2$-grading is defined by setting $\mathrm{deg}(x_i) = (1,0)$ and $\mathrm{deg}(y_i) = (0,1)$ for $i=0,1$.   An ideal $I$ of $S$ is a 
{\it bihomogeneous
ideal} if the generators of $I$ are bihomogeneous with
respect to this bigrading on $S$.  We use $\preceq$ to denote
the partial ordering on $\mathbb{N}^2$ defined by
$(i_1,j_1) \preceq (i_2,j_2)$ if  and only if
$i_1 \leq i_2$ and $j_1 \leq j_2$.

The original definition of virtual resolutions given in \cite{Virtual} uses the geometric language of sheaves. We do not give the sheaf-theoretic definition, but we will instead use an equivalent algebraic definition.  To state it, we need some notation.
Let $B$ denote the irrelevant ideal $\langle x_0, x_1\rangle \cap \langle y_0, y_1 \rangle = 
\langle x_0y_0,x_0y_1,x_1y_0,x_1y_1\rangle$ of $\mathbb{P}^1 \times \mathbb{P}^1$ in $S$. For any $S$-module $M$, let
$$  
\Gamma_B(M) := \{m \in M \, \mid \, B^t m = 0 \, \textup{ for some } t \in \mathbb{N} \},
$$
i.e., $\Gamma_B(M)$ consists of elements annihilated by some 
power of the irrelevant ideal. 
With this notation, we can state the definition for a complex to be a virtual resolution. Note that the homology groups $H_i(\mathcal{F}_{\bullet})$ of a complex $\mathcal{F}_{\bullet}$ of $S$-modules are naturally $S$-modules. 

\begin{definition}\label{virtualdefn}
Let $M$ be a finitely generated $S$-module. Let 
$$\mathcal{F}_{\bullet} := \dots \rightarrow F_2 \rightarrow F_1 \rightarrow F_0$$
be a complex of finitely generated free $S$-modules satisfying:
\begin{enumerate}
\item[(1)]for each $i>0$ there is some power $t$ such that $B^tH_i(\mathcal{F}_{\bullet})=0$, and
\item[(2)]$H_0(\mathcal{F}_{\bullet})/\Gamma_B(H_0(\mathcal{F}_{\bullet}))\cong M/\Gamma_B(M)$.
\end{enumerate}
Then $\mathcal{F}_{\bullet}$ is a {\it virtual resolution} of $M$.
\end{definition}

\begin{remark}
A proof of the equivalence of Definition \ref{virtualdefn}
and the original geometric definition for a virtual
resolution given in \cite{Virtual} can be found in Kennedy's thesis 
\cite[Theorem 4.9]{Jarvis}. 
\end{remark}

Given a bihomogeneous ideal $J$ of $S$, the 
{\it $B$-saturation} of $J$ is the 
ideal defined as 
$$J:B^\infty := \bigcup_{t \in \mathbb{N}} 
(J:B^t).$$
 An ideal $J$ is {\it $B$-saturated} if
$J:B^\infty = J$.
The following theorem gives
one way to construct a virtual resolution of $S/I$.  

\begin{theorem}
\label{virtualresideals}
Let $I$ be a $B$-saturated bihomogeneous ideal of $S$.
If $J$ is a bihomogeneous ideal of $S$ with
$J:B^\infty = I$, then the minimal free resolution of
$S/J$ is a virtual resolution of $S/I$.
\end{theorem}

\begin{proof}
If $\mathcal{F}_{\bullet}:= \cdots \rightarrow F_2 \rightarrow
F_1 \rightarrow F_0=S$ is a minimal free resolution of
$S/J$, then $H_i(\mathcal{F}_{\bullet}) = 0$ for all $i > 0$.
Since $\Gamma_B(S/J) = (J:B^\infty)/J = I/J$,
$$H_0(\mathcal{F}_{\bullet})/\Gamma_B(H_0(\mathcal{F}_{\bullet}))
= (S/J)/(I/J) \cong S/I.$$
Because $I$ is $B$-saturated, 
$\Gamma_B(S/I) = (I:B^\infty)/I = (0)$,
and thus $H_0(\mathcal{F}_{\bullet})/\Gamma_B(H_0(\mathcal{F}_{\bullet}))
\cong (S/I)/(\Gamma_B(S/I)).$  So $\mathcal{F}_{\bullet}$ is a
virtual resolution of $S/I$ by Definition \ref{virtualdefn}.
\end{proof}

For Section~\ref{trimmingsection}, we use a 
result of Berkesch, Erman, and Smith  \cite{Virtual} 
which constructs virtual resolutions out of a minimal free resolution $\mathcal{F}_{\bullet}$ by deleting certain summands from the $F_i$. 
The rough idea of this construction is that we can ``throw away'' some summands in $\mathcal{F}_{\bullet}$, and what remains will be a virtual resolution. We will (informally) refer to this process as {\it trimming} a minimal free resolution (to obtain a virtual one).

The notion of $\underline{d}$-regularity,
as introduced by Maclagan and Smith \cite{Reg},
is used in the statement of the next theorem.  We will not give the definition here because, for the purposes of this paper, it will suffice to use a characterization of the multigraded regularity for 
points in $\mathbb{P}^1 \times \mathbb{P}^1$ as given in
Theorem \ref{theorem: regular pts} below.
We now state a special case of the result in  \cite[Theorem 1.3]{Virtual}, which suffices for our results. 

\begin{theorem}\label{theorem:trim} 
Let $M$ be a finitely generated $\mathbb{Z}^2$-graded $B$-saturated $S$-module and let $\underline{d} \in \mathbb{Z}^2$. Suppose $M$ is $\underline{d}$-regular. If $G$ is the free subcomplex of a minimal free resolution of $M$ consisting of all summands generated in
degree at most $\underline{d}+(1,1)$ with respect to $\preceq$, then $G$ is a virtual resolution of $M$.
\end{theorem}

In the next section we will apply Theorem~\ref{theorem:trim} to the bihomogeneous coordinate rings $S/I_X$ where $X$ is a finite  set of points in $\mathbb{P}^1 \times \mathbb{P}^1$.  Note that going forward, we will assume that all sets
of points are finite.
As observed above, Theorem~\ref{theorem:trim} requires information about 
the {\it multigraded regularity} of $S/I_X$, denoted $\mathrm{reg}_B(S/I_X)$, defined to be the set of all $\underline{d}$ for which $S/I_X$ is $\underline{d}$-regular. 
Before giving the relevant statement, which is Theorem~\ref{theorem: regular pts}, we need some preliminaries and background concerning Hilbert functions and points in $\mathbb{P}^1 \times \mathbb{P}^1$.
See 
\cite{Adam} for details.

A point $P = [a_0: a_1] \times [b_0:b_1] \in \mathbb{P}^1 \times \mathbb{P}^1$ has bihomogeneous defining
ideal $I_P$ given by 
$$I_P = \langle a_1x_0 - a_0x_1,b_1y_0-b_0y_1 \rangle \subseteq S,$$
which is a prime ideal. For a set of $s$ distinct points $X = \{P_1,\ldots,P_s\}
\subseteq \mathbb{P}^1 \times \mathbb{P}^1$, the associated
bihomogenous ideal of $X$ is $I_X = I_{P_1} \cap \cdots \cap I_{P_s}$.
The coordinate ring of $X$ is $S/I_X$ (see \cite[Theorem 3.1, Corollary 3.2]{Adam} for more details).  Let K-$\dim (S/I_X)$ denote the Krull dimension of $(S/I_X)$. 
The following facts hold:

\begin{lemma}\label{P1xP1facts}
Let $X$ be a set of points in $\mathbb{P}^1 \times \mathbb{P}^1$
with coordinate ring $S/I_X$.
\begin{enumerate}
    \item There exists $L \in S$ with $\deg(L) = (0,1)$ 
    such that $\overline{L}$ is a non-zero-divisor in
    $S/I_X$.
    \item $1 \leq {\rm depth}(S/I_X) \leq \text{\emph{K-dim}}(S/I_X) = 2.$
    \item The ideal $I_X$ is $B$-saturated.
\end{enumerate}
\end{lemma}

\begin{proof}
For statement (1), \cite[Lemma 3.5]{Adam} proves the 
case $\deg(L)=(1,0)$ and the proof can be easily 
adapted to our case. 
Statement (2) is \cite[Lemma 4.2]{Adam}.
For statement (3), each ideal $I_{P_i}$ is $B$-saturated, so $I_X$ is $B$-saturated since the intersection
of $B$-saturated ideals is also saturated.
%For statement (3),  let $F$ be a form of %degree
%$(1,0)$ that does not pass through any first coordinate
%appearing in $X$.  Similarily, let $G$ %be a form
%of degree $(0,1)$ that satisfies $G(P) %\neq 0$ for all
%$P \in X$.   Note that $(FG)^t \in B^t$ %for all $t \geq 1$.   Now, suppose that %$H \in I_X:B^\infty$.  So,
%there is a $t$ such that $H(FG)^t \in %I_X$.  But
%for any $P\in X$, $(FG)(P) \neq 0$.  So $%H(P) = 0$.  Hence
%$H \in I_X$.  So $I_X:B^\infty \subseteq %I_X$.  The 
%other containment is immediate.
\end{proof}

\begin{remark}\label{convention}
By an appropriate change of coordinates, we can
assume without loss of generality that the non-zero-divisor
of Lemma \ref{P1xP1facts} (1) is $y_0$, i.e.,
we can assume that if $P = A \times B \in X$,
then $B \neq [0:1]$.
\end{remark}

We now recall the definition of Hilbert functions, their first and second differences, and some useful facts which we use below. 
First, for any bihomogeneous ideal $I \subseteq S$,
the {\it Hilbert function} $H_{S/I}$ of $S/I$ is the function $H_{S/I}:\mathbb{N}^2 \rightarrow \mathbb{N}$ defined by 
$$H_{S/I}(i,j):=\dim_{\Bbbk}(S/I)_{i,j}=\dim_{\Bbbk}S_{i,j}-\dim_{\Bbbk}[I]_{i,j}.$$

If $I_X$ is the defining ideal of a subvariety $X$ of $\mathbb{P}^1 \times 
\mathbb{P}^1$, we denote
$H_{S/I_X}$ by $H_X$.  
If we let $m_{i,j} := H_X(i,j)$ for $(i, j) \in \mathbb{N}^2$, we can
view $H_X$ as an infinite matrix $(H_X(i,j)) = (m_{i,j})_{(i, j) \in \mathbb{N}^2}$. The {\it first difference function} of $H_X$ is
the function $\Delta H_X: {\mathbb{N}}^2 \to \mathbb{N}$
given by 
$$   
\Delta H_X(i,j) = m_{i,j} - m_{i-1,j} - m_{i, j-1} + m_{i-1,j-1},
$$
where we use convention that $H(i,j) = m_{i,j} = 0$ if $i<0$ or $j<0$. If $c_{i,j} := \Delta H_X(i,j)$, then $\Delta H_X$ can also be
represented by the infinite matrix $(c_{i,j})_{(i,j) \in \mathbb{N}^2}$. 
The {\it second difference function} $\Delta^2 H_X$ is defined to be $\Delta(\Delta H_X)$, and so 
$$
\Delta^2 H_X(i,j) = c_{i,j} - c_{i-1,j} - c_{i, j-1} + c_{i-1, j-1}.
$$
We use the notation $\Delta^2 H_X = (d_{i,j})_{(i,j) \in \mathbb{N}^2}$ in a manner similar to $H_X$ and $\Delta H_X$. 

The functions $H_X$ and $\Delta^2 H_X$ capture
geometric  and algebraic information about a set of points $X$ 
in $\mathbb{P}^1 \times \mathbb{P}^1$.  
To explain this connection, we introduce
some further notation.    
Let $\pi_1:\mathbb{P}^1 \times \mathbb{P}^1 \rightarrow
\mathbb{P}^1$ denote the projection
map to the first coordinate.  Given a set 
of points $X \subseteq \mathbb{P}^1 \times \mathbb{P}^1$, the set $\pi_1(X) = \{A_1,\ldots,A_\ell\}$ is the set of distinct
first coordinates appearing in $X$.
For each $A_i \in \pi_1(X)$, let 
$\alpha_i = |\pi_1^{-1}(A_i) \cap X|$, i.e.,
the number of points of $X$ whose first coordinate
is $A_i$.  After relabelling, we can assume
that $\alpha_1 \geq \alpha_2 \geq \cdots \geq \alpha_\ell$.

We set $\alpha_X = (\alpha_1,\ldots,\alpha_\ell)$.
Note that the vector $\alpha_X$ is a partition
of the integer $|X|$, i.e., $|X| = \alpha_1+\cdots +\alpha_\ell$.  The {\it conjugate} of $\alpha_X$,
is the partition $\alpha^\star_X = 
(\alpha_1^\star,\ldots,\alpha^\star_{\alpha_1})$
where 

$\alpha_i^\star = |\{j \in \{1,\ldots,\ell\} ~|~
\alpha_j \geq i\}|$.  
%Here, we write
%$\alpha_j \in \alpha$ to mean that $\alpha_j$ 
%appears in the vector $\alpha$.  

We note that $|\pi_1(X)| =\alpha_1^\star \geq \alpha_2^\star \geq 
\cdots \geq \alpha^{\star}_{\alpha_1}$.
The conjugate $\alpha_X^\star$ of $\alpha_X$, which is related
to geometric information about $X$, is then 
encoded into the Hilbert function.  
 The next
theorem (see \cite[Theorem 3.29]{Adam}) 
describes one such relationship:

\begin{theorem}\label{specialvaluesofHX1}
Let $X \subseteq \mathbb{P}^1 \times \mathbb{P}^1$
be a set of points with
associated tuple $\alpha_X^\star 
= (\alpha_1^\star,\ldots,\alpha^\star_{\alpha_1})$.
%
%and let $L \in S$ correspond to the non-zero-divisor
%of Lemma \ref{P1xP1facts}.
For all $i \geq |\pi_1(X)|-1$, 
$$H_{S/I_X}(i,j) = \alpha_1^\star + \cdots +
\alpha_{j+1}^\star$$
%$$H_{S/(I_X + \langle L \rangle)}(i,j) = %\alpha_{j+1}^\star$$
where $\alpha_r^\star = 0$ if $r > \alpha_1$.
\end{theorem}

We will require the following corollary.

\begin{corollary}\label{specialvaluesofHX2}
Let $X \subseteq \mathbb{P}^1 \times \mathbb{P}^1$
be a set of points with
associated tuple $\alpha_X^\star 
= (\alpha_1^\star,\ldots,\alpha^\star_{\alpha_1})$, and let $L \in S$ correspond to the non-zero-divisor
of Lemma \ref{P1xP1facts}.
For all $i \geq |\pi_1(X)|-1$, 
$$H_{S/(I_X + \langle L \rangle)}(i,j) = \alpha_{j+1}^\star$$
where $\alpha_r^\star = 0$ if $r > \alpha_1$.
\end{corollary}

\begin{proof}
%By \cite[Theorem 3.29]{Adam}, for all $i \geq 
%|\pi_1(X)|-1$,
%$$H_{S/I_X}(i,j) = \alpha_1^\star + \cdots +
%\alpha_{j+1}^\star$$
%where $\alpha_r^\star = 0$ if $r \geq %\alpha_1$.
We have
a short exact sequence 
\[0 \rightarrow S/(I_X: \langle L \rangle)(0,-1)
\stackrel{\times \overline{L}}{\longrightarrow} 
S/I_X \rightarrow S/(I_X+\langle L \rangle) 
\rightarrow 0.\]
But $I_X:\langle L \rangle = I_X$ since $L$ is
a non-zero-divisor.  The maps in the
short exact sequence have degree $(0,0)$.
Thus, since Hilbert functions are additive on short
exact sequences,
$$H_{S/(I_X+\langle L \rangle)}(i,j) =
H_{S/I_X}(i,j) - H_{S/I_X}(i,j-1).$$
The conclusion follows from the formula for $H_{S/I_X}$ given 
in Theorem \ref{specialvaluesofHX1}.
\end{proof}

When $\alpha_X = (\alpha_1)$, i.e., the points of $X$ all have the same 
first coordinate,

\cite[Lemma 3.26]{Adam} 
gives more information about $I_X$
and its Hilbert function:

\begin{lemma}\label{pointsonruling}
Suppose $X \subseteq \mathbb{P}^1 \times \mathbb{P}^1$ is a set of points
of the form $$X  = \{A \times B_1, A\times B_2,\ldots, A\times B_\alpha\}.$$
with $A = [a_0:a_1]$ and $B_i = [b_{i,0}:b_{i,1}]$ for $i=1,\ldots,\alpha$.
Then
\begin{enumerate}
\item $H_X(i,j) = 
\begin{cases}
j+1 & \mbox{if $j < \alpha$} \\
\alpha & \mbox{if $j \geq  \alpha$.}
\end{cases}$
\item $I_X  = \langle L_A, L_{B_1}L_{B_2}\cdots L_{B_{\alpha}} \rangle$ 
where $L_A = a_1x_0-a_0x_1$ and
$L_{B_i} = b_{i,1}y_0-b_{i,0}y_1$ for $i=1,\ldots,\alpha$.
\end{enumerate}
\end{lemma}

\begin{remark}\label{swapruling}
By swapping the roles of the $A$'s and $B$'s and the roles of the
$i$'s and $j$'s, a similar result holds for points of the form
$X  = \{A_1 \times B, A_2\times B,\ldots, A_\alpha \times B\}.$

\end{remark}

We say that a set of points
$X \subseteq \mathbb{P}^1 \times \mathbb{P}^1$
has the {\it generic Hilbert function}
if
$$H_X(i,j) = \min\{|X|, (i+1)(j+1)\} 
~\mbox{for all $(i,j) \in \mathbb{N}^2$}.$$
As we saw above,
to use Theorem \ref{theorem:trim}, we need 
the multigraded
regularity of $S/I_X$.  If $X$ has the generic Hilbert function,
the multigraded regularity is easy to compute. The next result is simply a specialization
of \cite[Proposition 6.7]{Reg} to the case of 
points in $\mathbb{P}^1 \times \mathbb{P}^1$:

\begin{theorem}\label{theorem: regular pts}
If $X \subseteq \mathbb{P}^1 \times \mathbb{P}^1$
is a set of points with the generic Hilbert function, then 
$${\rm reg}_B(S/I_X) = \{ (i,j) ~|~ (i+1)(j+1) \geq |X| \}.$$
\end{theorem}

We now introduce some notation and results concerning the minimal free resolutions  of $S/I_X$. 
By Lemma \ref{P1xP1facts}, since $1 \leq {\rm depth}(S/I_X) \leq 2$, 
the Auslander-Buchsbaum Theorem (see \cite[Theorem 19.9]{Eisenbud}) implies that the $\mathbb{Z}^2$-graded minimal free resolution of 
$S/I_X$ has either length $2$ or $3$.  Following
the notation of Giuffrida, Maggioni, and Ragusa \cite{GMR2} we 
can write this resolution as 

\footnotesize
\begin{equation}\label{eq: min free res}
0 \rightarrow \bigoplus_{\ell=1}^p S(-a_{3,\ell},-a'_{3,\ell})\hookrightarrow \bigoplus_{\ell=1}^n S(-a_{2, \ell},-a'_{2, \ell}) \rightarrow \bigoplus_{\ell=1}^m S(-a_{1, \ell},-a'_{1, \ell}) \rightarrow S \rightarrow S/I_X \rightarrow 0
\end{equation}
\normalsize
With the above notation, we define 
$$
\alpha_{r,s}:=\# \lbrace (a_{1,i},a'_{1,i})=(r,s) \rbrace \, \textup{ and } \, 
\beta_{r,s}:=\# \lbrace (a_{2,i},a'_{2,i})=(r,s)\rbrace
$$
$$\textup{ and } \, 
\gamma_{r,s}:=\# \lbrace (a_{3,i},a'_{3,i})=(r,s)\rbrace. 
$$
Thus $\alpha_{r,s}$ is the number of 
generators of $I_X$ of degree $(r,s)$, and 
$\beta_{r,s}$ and $\gamma_{r,s}$ represent the number of
first and second syzygies of degree $(r,s)$, respectively. 
Giuffrida, Maggioni, and Ragusa \cite[Proposition 3.3]{GMR2} showed a number
of relations between $H_X,\Delta H_X$ and $\Delta^2 H_X$ and
the invariants in ~\eqref{eq: min free res}; we record
one such relation in the next theorem.

\begin{theorem}\label{prop:6} 
Let $X$ be a set of points in $\mathbb{P}^1 \times \mathbb{P}^1$ 
with second difference Hilbert function $\Delta^2 H_X = (d_{i,j})$.
If $(r,s)\succ (0,0)$, then $d_{r,s}= -\alpha_{r,s}+ \beta_{r,s}- \gamma_{r,s}.$
\end{theorem}

We will use another result of Giuffrida, Maggioni, and Ragusa, which (under a certain hypothesis on the set $X$ of points) relates the degrees of the minimal generators of $I_X$, in particular,
the values $\alpha_{i,j}$ defined above, to the integers $d_{i,j}$ in the second difference matrix $\Delta^2 H = (d_{i,j})$. The version stated here is a special case of  \cite[Theorem 4.3]{GMR1}:

\begin{theorem}\label{theorem:main theorem}
Let $s \geq 1$ be an integer. There exists a dense open subset $U$ of $(\mathbb{P}^1 \times \mathbb{P}^1)^s$ such that for every $(P_1,\dots,P_s) \in U$, the set of points $X = \lbrace P_1,\dots,P_s\rbrace$ satisfies:
\begin{enumerate} 
\item $X$ has generic Hilbert function, and
\item for any $(i,j) \in \mathbb{N}^2$ such that
$d_{i,j} <0$ and $d_{i,s} >0$ for some $s > j$, or,
$d_{i,j} < 0$ and $d_{r,j} > 0$ for some $r > i$, 
then we have $\alpha_{i,j}=-d_{i,j}$, i.e., the number of minimal generators of $I_X$ of degree $(i,j)$ is $-d_{i,j}$. 
\end{enumerate} 
Moreover, every non-zero $\alpha_{i,j}$ arises in the manner described in (2).  
\end{theorem}

Motivated by the above result,
a set of $s$ points $X= \lbrace P_1,\dots,P_s \rbrace$ in $\mathbb{P}^1 \times \mathbb{P}^1$ is said to be in {\it sufficiently general position} if $(P_1,\dots,P_s)$ belongs to the open set of 
Theorem \ref{theorem:main theorem}.

\begin{example}\label{runningexample}
We illustrate some of the ideas of this section with the following
running example.
Let $X$ be a set of six points in $\mathbb{P}^1 \times \mathbb{P}^1$
that are in sufficiently general position.   The Hilbert
function of $S/I_X$ is thus the generic Hilbert function.
Thus we know $H_X$, and consequently $\Delta^2 H_X$; these functions
are written as infinite matrices below (where rows and columns are index by $\{0,1,\ldots\}$):
\[
H_X = \begin{bmatrix} 
    1 & 2 & 3 & 4 & 5 & 6 & 6 & \  \\
    2 & 4 & 6 & 6 & 6 & 6 & 6 & \  \\
    3 & 6 & 6 & 6 & 6 & 6 & 6 & \ \\
    4 & 6 & 6 & 6 & 6 & 6 & 6 & \dots  \\
    5 & 6 & 6 & 6 & 6 & 6 & 6 & \  \\
    6 & 6 & 6 & 6 & 6 & 6 & 6 & \  \\
    6 & 6 & 6 & 6 & 6 & 6 & 6 & \  \\
    \ & \ & \ & \vdots & \ & \ & \   \\
    \end{bmatrix}
~~\mbox{and}~~
\Delta^2 H_X = \begin{bmatrix} 
    1 & 0 & 0 & 0 & 0 & 0 & -1 & \  \\
    0 & 0 & 0 & -2 & 0 & 0 & 2 & \  \\
    0 & 0 & -3 & 4 & 0 & 0 & -1 & \ \\
    0 & -2 & 4 & -2 & 0 & 0 & 0 & \dots  \\
    0 & 0 & 0 & 0 & 0 & 0 & 0 & \  \\
    0 & 0 & 0 & 0 & 0 & 0 & 0 & \  \\
    -1 & 2 & -1 & 0 & 0 & 0 & 0 & \  \\
    \ & \ & \ & \vdots & \ & \ & \   \\
    \end{bmatrix}
.\]
Theorem \ref{theorem:main theorem} says we can determine
the  number of the minimal generators of
$I_X$ of any particular bidegree by looking for 
negative entries in the matrix $\Delta^2 H_X = (d_{i,j})$ which also 
have a positive entry either to the right of, or below, the negative entry.
For our example of six points, these entries are:
$$\begin{array}{llll}
\bullet & d_{0,6}=-1 <0  ~~\mbox{and}~~ d_{1,6}=2>0 & \bullet & d_{3,1}=-2 <0
~~\mbox{and}~~ d_{6,1} =2>0 \\
\bullet & d_{1,3}=-2 <0 ~~\mbox{and}~~ d_{1,6}=2>0  & \bullet & 
d_{6,0} =-1 <0 ~~\mbox{and}~~  d_{6,1}=2>0. \\
\bullet & d_{2,2}=-3 <0 ~~\mbox{and}~~ d_{2,4}=4>0 &
\end{array}$$
So, the ideal $I_X$ has one minimal generator of degree $(0,6)$,
two minimal generators of degree $(1,3)$, three minimal
generators of degree $(2,2)$, two minimal generators of degree $(3,1)$,
and one minimal generator of degree $(6,0)$.
This agrees with the bigraded minimal free resolution
of $S/I_X$ obtained by using Macaulay2:
\begin{equation}\label{res6pts}
0 \rightarrow 
\begin{matrix}
S(-3,-3)^2\\
\oplus\\
S(-6,-2)\\
\oplus\\
S(-2,-6)
\end{matrix}
\rightarrow
\begin{matrix}
S(-3,-2)^4\\
\oplus\\
S(-2,-3)^4\\
\oplus\\
S(-1,-6)^2\\
\oplus\\
S(-6,-1)^2
\end{matrix} 
\rightarrow 
\begin{matrix}
S(-3,-1)^2
\oplus
S(-2,-2)^3\\
\oplus \\
S(-1,-3)^2\\
\oplus\\
S(0,-6)
\oplus
S(-6,0)
\end{matrix}
\rightarrow S \rightarrow S/I_X \rightarrow 0.
\end{equation}
\end{example}

Very roughly speaking, 
a set of $s$ points $X$ in
$\mathbb{P}^1 \times \mathbb{P}^1$ is in sufficiently general position if the points are ``random'' enough. In other words, there should be no subset $Y$ of $X$ that is contained  on some curve of smaller than
expected degree.   One way to interpret Theorem
\ref{theorem:main theorem} is that for 
a collection of points $X$ that is random enough, 
the number of generators of $I_X$ and their
degrees can be read off the Hilbert function.  As
the next example shows, simply having the correct
Hilbert function does not allow us to count
the generators and their degrees from $H_X$.

\begin{example}
Consider the five points 
$$X = \{[1:1]\times [1:1],[1:2]\times [1:2],[1:3]\ \times [1:3],[1:4] \times [1:4],[1:6] \times [1:8]\}.$$
These five points are not in sufficiently general position. Indeed, note that for this set of points, the 
bihomogeneous element $x_0y_1-x_1y_0$ of degree $(1,1)$ pass  through the first four points;  
for a sufficiently general set of points, we would expect
no four points of a set of five points to lie on a curve
of degree $(1,1)$.  

This set of points also gives us an example
of a set of points that satisfies the first condition
of Theorem \ref{theorem:main theorem}, but not
the second. The Hilbert function of $S/I_X$
is the generic Hilbert function, where
the corresponding second difference function is given
by
\[
H_X = \begin{bmatrix} 
    1 & 2 & 3 & 4 & 5 & 5 & \  \\
    2 & 4 & 5 & 5 & 5 & 5 & \  \\
    3 & 5 & 5 & 5 & 5 & 5 &  \ \\
    4 & 5 & 5 & 5 & 5 & 5 &  \dots  \\
    5 & 5 & 5 & 5 & 5 & 5 & \  \\
    5 & 5 & 5 & 5 & 5 & 5 & \  \\
    \ & \ & \ & \vdots & \ & \ & \   \\
    \end{bmatrix}
    ~~\mbox{and}~~
\Delta^2 H_X = \begin{bmatrix} 
    1 & 0 & 0 & 0 & 0 & -1 & \  \\
    0 & 0 & -1 & -1 & 0  & 2 & \  \\
    0 & -1 & 0 & 2 & 0  & -1 & \ \\
    0 & -1 & 2 & -1 & 0 & 0 & \dots  \\
    0 & 0 & 0 & 0 & 0 & 0 & \  \\
    -1 & 2 & -1 & 0 & 0 & 0  & \  \\
       \ & \ & \ & \vdots & \ & \\   \\
       \end{bmatrix}.\]
The ideal of $I_X$ contains a generator of degree $(2,2)$,
namely,
$$288x_0^2y_0^2-600x_0^2y_0y_1+41x_1^2y_0y_1+420x_0^2y_1^2-161x_0x_1y_1^2+12x_1^2y_1^2.$$
However, if $X$ was in sufficiently general position,
then the ideal $I_X$ would have no generator of degree $(2,2)$ by Theorem \ref{theorem:main theorem} 
because the $(2,2)$ entry 
of $\Delta^2 H_X$ is not negative. 

\end{example}

\begin{remark}
Gao, Li, Loper, and Mattoo \cite{Loper1} 
defined a set of points to 
be a {\it virtual complete intersection} if
$S/I_X$ has a virtual resolution of the 
form $K(f,g)$, where $K(f,g)$ denotes the 
Koszul complex of two bihomogeneous
forms $f$ and $g$.  Any finite
set of points $X$ in $\mathbb{P}^1 \times \mathbb{P}^1$ that is in
sufficiently general position is also
a virtual complete intersection.  To see
this fact, first note
that for such a set of points, 
the first coordinates, respectively
second coordinates, of $X$ are distinct,
so each horizontal and vertical ruling
of $\mathbb{P}^1 \times \mathbb{P}^1$ that
intersects $X$ meets at only one point.  Because
$X$ has the same number of points
on each vertical and horizontal ruling, 
$X$ is now a virtual complete intersection by
\cite[Theorem 3.1]{Loper1}.  In the next
section, we will give alternative virtual
resolutions for these sets of points.
\end{remark}

%%%%%%%%%%%%%%%%%%%%%%%%%%%%%%%%%%%%%%%

\section{Virtual resolutions via  ``trimming''}
\label{trimmingsection}

In this section, we give a family of explicitly described virtual resolutions for points in sufficiently general position
in $\mathbb{P}^1 \times \mathbb{P}^1$, using the trimming result of Berkesch, Erman, and Smith. Our main result is the following.

\begin{theorem}\label{maintheoremtrim}
Let $X$ be a set of points in 
$\mathbb{P}^1 \times \mathbb{P}^1$ in sufficiently general
position. Assume $\lvert X \rvert > 1$ and suppose that
there exists an $(i,j) \in \mathbb{N}^2$ with $i < j$ such that
$|X| = (i+1)(j+1)$.   Then there is a virtual resolution of $S/I_X$ of length two.
In fact, the virtual resolution can be described explicitly
as follows.
Let $|X| = (i+2)q+r$ with $q, r \in {\mathbb N}$ and $0 \leq r < (i+2)$. 

Then the virtual resolution has the form

$$0 \rightarrow 
S(-i-1,-j-1)^{2i+2} \rightarrow 
\begin{array}{c}
S(-i,-j-1)^{i+1} \\
    \oplus \\
S(-i-1,-q)^{i+2-r}       \\
      \oplus \\
S(-i-1,-q-1)^r
\end{array}
\rightarrow S \rightarrow S/I_X \rightarrow 0,
$$
where if $r=0$, the module $S(-i-1,-q-1)^0$ does not appear.
\end{theorem}

\begin{remark}
If $|X|=1$, then $I_X$ is a complete intersection and the minimal free resolution of $S/I_X$ 
$$0\rightarrow S(-1,-1)\rightarrow S(-1,0)\oplus S(0,-1)\rightarrow S\rightarrow S/I_X \rightarrow 0 $$
is already of length two. For this reason, we omit this case in the theorem above. 
\end{remark}

\begin{remark}
For a set $X$ of points with $\lvert X \rvert >1$ in sufficiently general position in 
$\mathbb{P}^1 \times \mathbb{P}^1$, there
always exists at least one tuple $(i,j)$ that satisfies the hypotheses, namely
$(i,j) = (0,|X|-1)$.  
\end{remark}

As a first step toward the proof of Theorem~\ref{maintheoremtrim}, we show that its hypotheses 
completely determine a submatrix of the second difference Hilbert function $\Delta^2 H_X$. In fact, we can compute this submatrix quite explicitly, as recorded in the lemma below. 

\begin{lemma}\label{submatrix}
Assume the hypotheses and notation of Theorem \ref{maintheoremtrim}.  

If $i> 0$ and $q<j$, then
the $(i+2) \times (j+2)$ submatrix of $\Delta^2 H_X$ consisting of 
the first $(i+2)$ rows and the first $(j+2)$ columns of $\Delta^2 H_X$ is of the form:
$$
\bordermatrix{ 
  & 0 & 1 & \cdots & q & q+1 & \cdots & j & j+1 \cr 
0 & 1 & 0 & \cdots & 0    & 0& \cdots & 0  & 0 \cr
1& 0 & 0 & \cdots & 0    & 0& \cdots & 0  & 0 \cr
\vdots & \vdots & \vdots & & \vdots & \vdots & & \vdots & \vdots \cr
i-1 & 0 & 0 & \cdots & 0    & 0& \cdots & 0  & 0 \cr
i   & 0 & 0 & \cdots & 0    & 0& \cdots & 0  & -i-1 \cr
i+1 & 0 & 0 & \cdots & -i-2+r    & -r& \cdots & 0  & 2i+2 \cr
}.$$

{Note that  we take the convention that if $j=q+1$, then we ignore the column labelled $j$ and take the column labelled $q+1$, and if $i=1$, then we ignore the row labelled $i-1$ and take the row labelled $0$. }

If $i>0$ and $q=j=i+1$, then the $(i+2) \times (i+3)$ submatrix of $\Delta^2 H_X$ consisting of the first $(i+2)$ rows and the first $(i+3)=(j+2)$ columns of $\Delta^2 H_X$ is of the form: 
$$
\bordermatrix{ 
  & 0 & 1 & \cdots & j-2 & j-1  & j & j+1 \cr 
0 & 1 & 0 & \cdots & 0 & 0   &   0  & 0  \cr
1& 0 & 0 & \cdots & 0  & 0   & 0  & 0 \cr
\vdots & \vdots & \vdots & & \vdots & & \vdots & \vdots \cr
i-1 & 0 & 0 & \cdots & 0    & 0 & 0  & 0 \cr
i   & 0 & 0 & \cdots & 0    & 0 & 0  & -i-1 \cr
i+1 & 0 & 0 & \cdots & 0    & 0 & -i-2 & 2i+2 \cr
}.$$

{We make the convention that if $i=1$, then we ignore the row labelled $i-1$ and take the row labelled $0$.}

Finally, if $i=0$, then the first two rows and $j+2$ columns of $\Delta^2 H_X$ are of the form: 
$$\bordermatrix{ 
  & 0 & 1 & \cdots & q-1 & q & \cdots & j & j+1 \cr 
0   & 1 & 0 & \cdots & 0    & 0& \cdots & 0  & -1 \cr
1 & 0 & 0 & \cdots & -2+r    & -r& \cdots & 0  & 2 \cr
}.$$
\end{lemma}

\begin{proof}
We first consider the case  $i > 0$ and $q<j$. Note that $i(j+2) < |X| = (i+1)(j+1)$.  
Indeed, $i(j+2) \geq (i+1)(j+1)$ implies that
$i \geq j+1$, contradicting our assumption that $i < j$.  So, because $X$ has the
generic Hilbert function, $H_X(i-1,j+1) = i(j+2)$.

More generally, under the assumptions $i>0$ and $q<j$ we can compute the 
first $(i+2)$ rows and first $(j+2)$ columns of $H_X$ to be the following: 
\footnotesize
$$\bordermatrix{ 
  & 0 & 1 & \cdots & q-1 & q & \cdots &j-1 & j & j+1 \cr 
0 & 1 & 2 & \cdots & q    & q+1& \cdots & j & j+1  & j+2 \cr
1& 2 & 4 & \cdots & 2q    & 2(q+1)& \cdots &2j & 2(j+1)  & 2(j+2)\cr
\vdots & \vdots & \vdots & & \vdots & \vdots & \vdots & \vdots & \vdots & \vdots \cr
i-1 & i &  (i)2 &\cdots & iq    & i(q+1)& \cdots  &ij &i(j+1)  & i(j+2) \cr
i   & i+1 & (i+1)2& \cdots & (i+1)q    & (i+1)(q+1)& \cdots & (i+1)j & |X|  & |X| \cr
i+1 & i+2 & (i+2)2 & \cdots & (i+2)q  & |X| & \cdots & |X| & |X| & |X| \cr
}.$$
\normalsize
Note we are using the fact that our hypotheses imply that 
$$H_X(i+1,q-1) = \min\{(i+2)q,|X|\} = 
\min\{(i+2)q,(i+2)q+r\} = (i+2)q$$
and that $H_X(i+1,q)=\lvert X \rvert.$
The conclusion of the lemma for this special case now follows from the definition of $\Delta^2 H_X$ and the above description of $H_X$.

The proofs for the case that $i>0$ and $q=j$
and the case that $i=0$ are similar and left to the reader.   For the case $i>0$ and $q=j$, note that since $i < j$,
$|X| = (i+2)q+r = (i+1)(j+1)$ only occurs if $r=0$ and $q = j = i+1$.
\end{proof}

Using the notation introduced in Section \ref{backgroundsection}, we require
the following calculations that compute the number of generators, syzygies, and
second syzygies of $I_X$ for specific degrees.

\begin{lemma}\label{gendegreelemma}
Assume the hypotheses and notation of Theorem \ref{maintheoremtrim},
and let $\Delta^2 H_X = (d_{r,s})$.  If
$(a,b) \preceq (i+1,j+1)$, then
\begin{enumerate}
    \item $\alpha_{a,b}= - d_{a,b}$ if and only if $d_{a,b}<0$, and $\alpha_{a,b}=0$ otherwise

\item
$
\beta_{a,b} = 
\begin{cases}
2i+2 & \mbox{if $(a,b) = (i+1,j+1)$} \\
0 & \mbox{otherwise}
\end{cases}
$
\item $\gamma_{a,b} = 0$.
\end{enumerate}
\end{lemma}

\begin{proof}
We first prove the formula for $\alpha_{a,b}$, the number of minimal generators of $I_X$
of degree $(a,b)$. Because $X$ is in sufficiently 
general position, by Theorem \ref{theorem:main theorem}, $\alpha_{a,b}$ can be read
off of $\Delta^2 H_X$.  Since we are only interested in all $(a,b) \preceq (i+1,j+1)$, we can
use Lemma \ref{submatrix} to compute all $\alpha_{a,b}$ in this range.  In particular, we note that 
for each $d_{a,b}$ with $d_{a,b}<0$ in this range, 
there exists an entry in $\Delta^2 H_X$ that is either to the right or below it which is positive. 
Applying Theorem \ref{theorem:main theorem} then gives the desired result.

To prove the formulas for $\beta_{a,b}$ and $\gamma_{a,b}$, we remind the reader that 
$$d_{r,s} = -\alpha_{r,s}+\beta_{r,s}- \gamma_{r,s} ~~\mbox{for all $(0,0) \neq (r,s)\in \mathbb{N}^2$}$$
by Theorem~\ref{prop:6}. 
We also recall the general fact that if $\beta_{r,s} >0$ (respectively $\gamma_{r,s} >0)$, then
there is some $\alpha_{a,b} >0$ (respectively, $\beta_{a,b} >0$) with $(a,b) \prec (r,s)$.
This is because $\beta_{r,s}$ records a syzygy among generators of degree strictly less than
$(r,s)$, and similarly, $\gamma_{r,s}$ records a second syzygy among the syzygies of
degrees strictly less than $(r,s)$.

For the discussion below, we always restrict to $(a,b) \preceq (i+1,j+1)$. 
As we have already observed, $\alpha_{a,b}>0$ exactly at the locations $(a,b)$ such that $d_{a,b}<0$  and $\alpha_{a,b}=0$ elsewhere. Let $(a,b)$ be such that $\alpha_{a,b}>0$ and assume that $b$ is the smallest index with $\alpha_{a,b}>0$. An examination of the different cases of $\Delta^2 H_X$ in Lemma~\ref{submatrix} yields the conclusion that $\alpha_{u,v}=0$ for any $(u,v) \prec (a,b)$, i.e., $I_X$ has no generators of these degrees. This then implies that $\beta_{u,v}=0$ and hence also $\gamma_{u,v}=0$ for such $(u,v)$. 

We now examine the locations at which $\alpha_{a,b}>0$ more closely. 

For $(a,b)=(i,j+1)$, we have
$$-\alpha_{i,j+1} = d_{i,j+1} = -\alpha_{i,j+1}+\beta_{i,j+1}-\gamma_{i,j+1}.$$
So, $\beta_{i,j+1}= \gamma_{i,j+1}$.  But now $\beta_{i,j+1}=0$ 
since, as we saw above, there is no $\alpha_{a,b} > 0$ with $(a,b) \prec (i,j+1)$.  Consequently,
$\gamma_{i,j+1}=0$.  
This argument also works to show that $\beta_{i+1,v}=\gamma_{i+1,v}=0$ as well, for $v$ the minimal index with $\alpha_{i+1,v} >0$. 

Suppose we are in a case when $\alpha_{i+1,v+1} >0$ where $v$ is the minimal index as above. 
For $(a,b)=(i+1,v+1)$, we again have 
$$-\alpha_{i+1,v+1} = d_{i+1,v+1} = -\alpha_{i+1,v+1}+\beta_{i+1,v+1}-\gamma_{i+1,v+1},$$
i.e., $\beta_{i+1,v+1} = \gamma_{i+1,v+1}$.  But since we have now seen that there is no $\beta_{a,b} >0$ for any
$(a,b) \prec (i+1,v+1)$, we must have $\gamma_{i+1,d+1}=0$.
But this implies $\beta_{i+1,v+1}=0$ also. 

Looking again at the possibilities for $\Delta^2 H_X$ as given in Lemma~\ref{submatrix}, we see that it remains to show that $\beta_{i+1,k}=\gamma_{i+1,k}=0$ for some values of $k \leq j+1$. We may proceed inductively. Let $k$ be the minimal index for which we wish to prove the claim and assume $k<j+1$ so that $d_{i+1,k}=0$.  We already know $\alpha_{i+1,k}=0$ and hence, by a similar argument as that above, $\beta_{i+1,k}=\gamma_{i+1,k}$. But we have already shown that all $\beta_{a,b}=0$ for $(a,b)\prec (i+1,k)$ and hence $\gamma_{i+1,k}=0$, which implies $\beta_{i+1,k}=0$. Now by induction the same holds for $\beta_{i+1,k+1}$ and $\gamma_{i+1,k+1}$ and so on, provided that $k<j$. 

Finally, we must analyze the case $(a,b)=(i+1,j+1)$, where we have $d_{i+1,j+1}=2i+2$. Since $\alpha_{i+1,j+1}=0$, we have
$$2i+2 = d_{i+1,j+1} = \beta_{i+1,j+1}-\gamma_{i+1,j+1}.$$
But as before, we have already seen that there is no $\beta_{a,b} >0$ for any $(a,b) \prec (i+1,j+1)$.  Hence
$\gamma_{i+1,j+1}=0$ and $\beta_{i+1,j+1}=2i+2$, as desired.  This completes the proof.
\end{proof}

We now come to the proof of Theorem \ref{maintheoremtrim}.

\begin{proof}[Proof of Theorem \ref{maintheoremtrim}]
Our hypotheses imply that $H_X(i,j) = \min\{(i+1)(j+1),|X|\} = (i+1)(j+1)$.  Thus, by
Theorem \ref{theorem: regular pts}, we have that $(i,j) \in {\rm reg}_B(S/I_X)$.  Hence,
when we trim the bigraded minimal free resolution of $S/I_X$ in degrees at most
$(i+1,j+1)$ using the partial
order $\preceq$, we obtain a virtual resolution of $S/I_X$ by Theorem 
\ref{theorem:trim}.

The bigraded minimal free resolution of $S/I_X$  can be written as
$$0 \rightarrow  \bigoplus_{(r,s) \in \mathbb{N}^2} S(-r,-s)^{\gamma_{r,s}}  \rightarrow  \bigoplus_{(r,s) \in \mathbb{N}^2} S(-r,-s)^{\beta_{r,s}} \rightarrow \bigoplus_{(r,s) \in \mathbb{N}^2} S(-r,-s)^{\alpha_{r,s}} \rightarrow S 
\rightarrow S/I_X \rightarrow 0.$$
If we now trim this resolution with respect to $(i+1,j+1)$, i.e., we keep those
summands generated in degrees $\preceq (i+1,j+1)$, then the conclusion follows from Lemma \ref{gendegreelemma} 
which explicitly calculates
$\alpha_{a,b}, \beta_{a,b}$ and $\gamma_{a,b}$ for all $(a,b) \preceq (i+1,j+1)$.
\end{proof}

\begin{example}
Suppose $X$ is a set of $6$ points in
sufficiently general position in ${\mathbb P}^1 \times {\mathbb P}^1$. Then the associated Hilbert function and second difference function can be found in Example \ref{runningexample}.
We first note that $|X| = (1+1)(2+1)$, so we will take $(i,j) = (1,2)$.  We also 
have $|X| =2\cdot 3 + 0$, so $r=0$ and $q=2$ in this case.  So, 
by  Theorem \ref{maintheoremtrim},

$$0 \rightarrow S(-2,-3)^4 \rightarrow S(-2,-2)^3\oplus S(-1,-3)^2
\rightarrow S \rightarrow S/I_X \rightarrow 0$$

is a virtual resolution of $S/I_X$.
To see why this is true in this
special case, note that $H_X(1,2) = 6$ so $(1,2) \in {\rm reg}_B(S/I_X)$.
We are thus trimming the resolution of $S/I_X$ in Example \ref{runningexample}
using $(1,2)+(1,1) =(2,3)$.
\end{example}

\begin{remark}
Because of the symmetry of $H_X$, we can make a statement 
similar to Theorem \ref{maintheoremtrim} that assumes
$j < i$.  In this case, the bidegrees need to be swapped and the roles
of $i$ and $j$ are also swapped.
\end{remark}

\begin{remark}
Computer experiments with Macaulay2 suggest that one might be able
to get additional virtual resolutions from $\Delta^2 H_X$.  In particular,
we have observed that if $X$ is in sufficiently general position 
with  $\Delta^2 H_X = (d_{i,j})$, and   
if $$(i,j) \in \min\{(a,b) ~|~ H_X(a,b) = |X| \},$$ then there exists a virtual resolution of $S/I_X$ of length two of the following form: 

\footnotesize
\begin{multline*}
0 \rightarrow 
\bigoplus_{
\begin{array}{c}
(0,0) \prec (a,b) \preceq (i+1,j+1) \\
d_{a,b} > 0
\end{array}}
S(-a,-b)^{d_{a,b}} \rightarrow
\\
\bigoplus_{
\begin{array}{c}
(a,b)  \preceq (i+1,j+1) \\
d_{a,b} < 0     
\end{array}}
S(-a,-b)^{|d_{a,b}|} 
\rightarrow S \rightarrow S/I_X \rightarrow 0.
\end{multline*}
\normalsize

Note that if $(i+1)(j+1) = |X|$, then $(i,j) \in \min\{(a,b) ~|~ H_X(a,b) = |X| \}$.  The above formula agrees with
Theorem
\ref{maintheoremtrim}.
\end{remark}
%%%%%%%%%%%%%%%%%%%%%%%%%%%%%%%%%%%%%%%%%%%%%%%%%%%%%

\section{Bounding a result of Berkesch, Erman, and Smith}
\label{boundsection}

In Section \ref{trimmingsection} we used 
Berkesch, Erman, and Smith's 
trimming construction
to find a short virtual resolution of $S/I_X$ for a
set of points $X$ in 
sufficiently general position in $\mathbb{P}^1 \times \mathbb{P}^1$. 
For zero-dimensional schemes 
in $\mathbb{P}^{n_1} \times \cdots 
\times \mathbb{P}^{n_r}$,
Berkesch, Erman, and Smith 
also proved the existence of
short virtual resolutions using the 
minimal free resolution of a specific module.
 We state below a special case of their
result (see \cite[Theorem 4.1]{Virtual}) for the case of points in $\mathbb{P}^1 \times  \mathbb{P}^1$.

\begin{theorem}\label{theorem: shorten}
Let $X$ be a set of points in $\mathbb{P}^1 \times
\mathbb{P}^1$.  Then for all $a \gg 0$, the 
minimal free resolution of 
$S/(I_X \cap \langle x_0,x_1 \rangle^a)$ is a virtual
resolution of $S/I_X$ of length two.
\end{theorem}

The purpose of this section is to improve the above
result of Berkesch, Erman, and Smith by giving an explicit bound on the $a$ that satisfies the claim in Theorem~\ref{theorem: shorten}.   
Specifically, we prove the following result. (Recall that 
$\pi_1(X)$ denotes the set of distinct first coordinates appearing in $X$.)

\begin{theorem}\label{maintheorem2}
Let $X$ be a set of points in $\mathbb{P}^1 \times 
\mathbb{P}^1$.  For all $a \geq |\pi_1(X)|-1$, the
minimal free resolution of $S/(I_X \cap \langle x_0,x_1 \rangle^a)$
is a virtual resolution of $S/I_X$ of length two.
\end{theorem}

The key lemma, whose proof we postpone until the
end of this section, is the following ideal
decomposition.

\begin{lemma}\label{keylemma}
Let $X$ be any set of points in $\mathbb{P}^1 \times 
\mathbb{P}^1$ with $\pi_1(X)  =
\{A_1,\ldots,A_\ell\}$ and $\alpha_k 
:= |\pi_1^{-1}(A_k) \cap X|$.  Suppose that 
$y_0$ is a non-zero-divisor on $S/I_X$.
Then for any integer $a \geq |\pi_1(X)|-1$, 
the equality 
\begin{equation}\label{eq: keylemma}
\langle I_X \cap \langle x_0,x_1 \rangle^a,y_0 \rangle = \bigcap_{A_k \in \pi_1(X)}
\langle y_0,y_1^{\alpha_k},L_{A_k} \rangle \cap \langle \langle x_0,x_1\rangle ^a,y_0\rangle
\end{equation}
is a primary decomposition of $\langle I_X \cap \langle x_0,x_1 \rangle^a,y_0 \rangle$, 
where $L_{A_k} = a_{k,1}x_0-a_{k,0}x_1$ if $A_k = [a_{k,0}:a_{k,1}]$.
\end{lemma}

Assuming the above lemma, we prove our main result of this section, Theorem~\ref{maintheorem2}. 

\begin{proof}[Proof of Theorem \ref{maintheorem2}]

We first claim that for all integers $a$, the
minimal free resolution of 
$S/(I_X \cap \langle x_0,x_1\rangle^a)$ is a 
virtual resolution of $S/I_X$.
To see this, note that 
by Lemma \ref{P1xP1facts}, the ideal $I_X$ is
$B$-saturated.   Moreover, we have
the identity
\begin{equation}\label{eq: B-sat of IX}
(I_X \cap \langle x_0,x_1 \rangle^a):B^\infty = I_X
~~\mbox{for all integers $a \geq 0$}.
\end{equation}
Indeed, because $(I_X \cap \langle x_0,x_1\rangle^a) \subseteq I_X$, we get
$(I_X \cap \langle x_0,x_1 \rangle^a):B^\infty 
\subseteq I_X:B^\infty = I_X,$
where we are using the fact that
$I_X$ is $B$-saturated by Lemma \ref{P1xP1facts} (3).  For the reverse inclusion, it is enough to note that
$I_XB^a \subseteq I_X \cap B^a 
\subseteq (I_X \cap \langle x_0,x_1 \rangle^a)$.  Our claim
now follows from Theorem \ref{virtualresideals}.

%Let $F$ and $G$ be as in the proof of Lemma
%\ref{P1xP1facts} (3).  
%If $H \in (I_X \cap \langle x_0,x_1 %\rangle^a):B^\infty$, then $H \in (I_X %\cap \langle x_0, x_1 \rangle^a):B^t$ for some $t$, i.e. 
%$H(FG)^t \in (I_X \cap \langle x_0,x_1 %\rangle^a)
%\subseteq I_X$.  As we argued in Lemma %\ref{P1xP1facts},
%this implies that $H \in I_X$. 
%Conversely, for any $H \in I_X$, we have 
%$HB^a \subseteq I_X$.  But since
%$B^a \subseteq \langle x_0,x_1 \rangle^a$, we %have
%$HB^a \subseteq I_X \cap \langle x_0,x_1 \rangle^a$,
%or equivalently, $H \in (I_X \cap \langle %x_0,x_1 
%\rangle^a):B^a$. Thus~\eqref{eq: B-sat of IX} holds and the claim is proved. 

It now suffices to show that for any $a \geq \lvert \pi_1(X) \rvert -1$, the minimal free resolution of $S/(I_X \cap \langle x_0, x_1 \rangle^a)$ has length $2$.
We first note that 
the Krull dimension of $S/(I_X \cap \langle x_0,x_1 \rangle^a)$ is $2$.  This can be seen geometrically, as follows. If we
ignore the bigrading, the zero-locus $Z = Z(I_X \cap \langle x_0,x_1 \rangle^a)$ defines an algebraic set in $\mathbb{A}^4$.  By
the Nulstellenstaz, 
$$Z = Z(\sqrt{I_X \cap \langle x_0,x_1 \rangle^a}) = Z(\sqrt{I_X} \cap
\sqrt{\langle x_0,x_1 \rangle^a}) = Z(I_X \cap \langle x_0,x_1 \rangle)$$
since $I_X = \sqrt{I_X} = \bigcap_{i=1}^s I_{P_i}$ is the intersection of prime ideals.
The ideals $I_{P_i}$ for 
$i=1,\ldots,s$ and the ideal
$\langle x_0,x_1 \rangle$ define planes in $\mathbb{A}^4$.  So,
$Z$ is union of planes in $\mathbb{A}^4$.  Hence
$\text{K-dim} (S/(I_X \cap \langle x_0,x_1 \rangle^a)) = \dim Z = 2$.
Because ${\rm depth}(M) \leq \mbox{K-dim}(M)$ for any $S$-module $M$
(e.g., see \cite[Lemma 2.3.6]{V}), we conclude that ${\rm depth}(S/(I_X \cap \langle x_0,x_1 \rangle^a) \leq 2$.  

To complete the
proof, it suffices to show that for all 
$a \geq |\pi_1(X)| -1$, the depth is exactly two.
Indeed, if we can show this fact, then the
Auslander-Buchsbaum Theorem implies that the 
projective dimension of $S/(I_X \cap \langle x_0,x_1 \rangle^a)$ is two, implying (by definition) that its minimal free resolution is of length $2$, as desired. 

We now show that for 
all 
$a \geq |\pi_1(X)| -1$, the depth is exactly two.
By Remark \ref{convention}, we can assume
that $y_0$ is a non-zero-divisor of $S/I_X$.  
In fact, $y_0$ is also a non-zero-divisor of
$S/(I_X \cap \langle x_0,x_1\rangle^a)$ because
if $y_0H \in (I_X \cap \langle x_0,x_1 \rangle^a)$,
we have $y_0H \in I_X$ and $y_0H \in \langle x_0,x_1 \rangle^a$, and this can only happen if 
$H \in I_X \cap \langle x_0,x_1 \rangle^a$.

Next, we claim that the depth of 
$S/(I_X \cap \langle x_0, x_1 \rangle^a, y_0)$ is not zero. 
This would show that there is a regular sequence of at least $2$ in $S/(I_X \cap \langle x_0, x_1 \rangle^a)$ (the first element of which is $y_0$), which would prove that the depth is $\geq 2$ as claimed. 
Because $a \geq |\pi_1(X)|-1$, Lemma \ref{keylemma} gives a primary decomposition of $\langle I_X \cap \langle x_0, x_1 \rangle^a,y_0\rangle$.
The zerodivisors of
$S/\langle I_X \cap \langle x_0, x_1 \rangle^a,y_0\rangle$ are 
the elements in the union of the associated primes (e.g.,
see \cite[Lemma 2.1.19]{V}), that is, the elements of the set
\begin{equation}\label{nzd}
\left(\bigcup_{A_k \in \pi_1(X)} \langle y_0,y_1,L_{A_k} \rangle
\right) \cup \langle x_0,x_1,y_0 \rangle.
\end{equation}
Let $L$ be any element with $\deg L = (1,0)$ such that $L$ does
not vanish at any point in $\pi_1(X)$.  Then $L+y_1$ is not
in the union ~\eqref{nzd}, so $L+y_1$ is a non-zero-divisor 
of $S/\langle I_X \cap \langle x_0, x_1 \rangle^a,y_0\rangle$.
 As we saw above, this implies that ${\rm depth}(S/(I_X \cap \langle x_0, x_1 \rangle^a)) \geq 2$. Putting this together with the fact that 
${\rm depth}(S/(I_X \cap \langle x_0,x_1 \rangle^a)) \leq \mbox{K-dim} (S/(I_X \cap \langle x_0,x_1 \rangle^a)) =2$  we obtain that the depth is $2$, as desired. This now completes the proof of the theorem. 
\end{proof}

It remains to prove Lemma \ref{keylemma}, which we now do.

\begin{proof}[Proof of Lemma \ref{keylemma}] 
We start with some observations about our set of points $X$ and associated defining ideal
$I_X$. By definition of the $\alpha_k$'s, for 
each $A_k \in \pi_1(X)$,  there
exist $\alpha_k$ many points $\{B_{k,1}, \ldots, B_{k,\alpha_k}\}$ in $\mathbb{P}^1$ such that
$$X_k:= \{A_k \times B_{k,1},A_k \times B_{k,2}, \ldots, A_k \times B_{k, \alpha_k}\} \subseteq
X$$
and all other points in $X$ have first coordinate differing from $A_k$. Thus $X_k$ is exactly the subset of $X$ consisting of points with first coordinate equal to $A_k$. 
Consequently, $I_X = \bigcap_{k=1}^\ell I_{X_k}$.  By
Lemma \ref{pointsonruling}, for each $k$, $I_{X_k} = \langle F_k, G_k \rangle$ where
$\deg F_k = (1,0)$ and $\deg G_k = (0,\alpha_k)$.  Furthermore, by Remark \ref{convention},
we can assume that no second coordinate has the form $[0:1]$, and consequently,
by Lemma \ref{pointsonruling} we have
$y_0\nmid G_k$ for all $k$.

In general, for any three ideals $I,J,K$, we always have $(I \cap J)+K \subseteq (I+K) \cap (J+K)$.  
It follows that for all integers $a \geq 0$
\begin{equation}\label{eq: distribute}
\begin{split} 
\langle I_X \cap \langle x_0,x_1 \rangle^a, y_0 \rangle & = 
\langle I_{X_1} \cap \cdots \cap  I_{X_\ell} \cap \langle x_0,x_1 \rangle^a, y_0 \rangle \\
& \subseteq  \langle I_{X_1},y_0 \rangle \cap \cdots \cap \langle I_{X_\ell},y_0 \rangle \cap
\langle \langle x_0,x_1 \rangle^a, y_0 \rangle.
\end{split} 
\end{equation}
From Lemma~\ref{pointsonruling} we know that for any $k$ we have 
\begin{equation}\label{eq: IXk y0}
\langle I_{X_k},y_0 \rangle = \langle L_{A_k},L_{B_{k,1}}\cdots L_{B_{k, \alpha_k}},y_0 \rangle
= \langle y_0,y_1^{\alpha_k}, L_{A_k} \rangle
\end{equation}
where the second equality holds because 
$L_{B_{k,1}}\cdots L_{B_{k,\alpha_k}}$ is a homogeneous 
polynomial of degree $\alpha_k$ only in the variables $y_0$ and $y_1$, and the
coefficient of $y_1^{\alpha_k}$ is not zero.  Consequently, combining~\eqref{eq: distribute} and~\eqref{eq: IXk y0}, it follows that the 
LHS of~\eqref{eq: keylemma} is contained in its RHS, for any $a \geq 0$. 

We now wish to show the reverse containment, i.e. the RHS of~\eqref{eq: keylemma} is contained in the LHS, for all $a \geq \lvert \pi_1(X) \rvert -1$. 
Since both ideals are bihomogeneous, it suffices to show the inclusion on each $(i,j)$-th graded piece, that is, to show that
\begin{equation}\label{eq: reverse inclusion}
\left[ \bigcap_{A_k \in \pi_1(X)}
\langle y_0,y_1^{\alpha_i},L_{A_k} \rangle \cap \langle \langle x_0,x_1 \rangle^a,y_0\rangle\right]_{i,j} \subseteq \left[\langle I_X \cap \langle x_0,x_1 \rangle^a,y_0 \rangle \right]_{i,j}
\end{equation}
for all $(i,j) \in \mathbb{N}^2$.   We now consider two cases: (1) $0 \leq i <a$ and (2) $a \leq i$.  

We first argue for case (1). Suppose that $0 \leq i < a$ and suppose that $F$ lies in the LHS of~\eqref{eq: reverse inclusion}. 
In particular, $F \in \langle \langle x_0,x_1 \rangle^a,y_0 \rangle$.  
Since $0 \leq i<a$ by assumption, we conclude that $F \in \langle y_0 \rangle$. But then $F \in \left[\langle I_X \cap \langle x_0,x_1 \rangle^a,y_0 \rangle \right]_{i,j}$ also, so~\eqref{eq: reverse inclusion} holds. 

For case (2), we have $a \leq i$, so $\left[\langle x_0,x_1 \rangle^a\right]_{i,j} =S_{i,j}$.  Consequently  
$$\langle \langle x_0,x_1 \rangle^a,y_0 \rangle_{i,j} = S_{i,j}
~~\mbox{and}
\left[ \langle I_X \cap \langle x_0,x_1 \rangle^a,y_0 \rangle \right]_{i,j} = 
\left[\langle I_X,y_0 \rangle\right]_{i,j}.$$
Thus, to prove~\eqref{eq: reverse inclusion} in this case it suffices to check that
\begin{equation}\label{eq: rev incl take 2} 
\left[ \bigcap_{A_k \in \pi_1(X)}
\langle y_0,y_1^{\alpha_k},L_{A_k} \rangle \right]_{i,j} \subseteq \left[\langle I_X,y_0 \rangle \right]_{i,j}.
\end{equation}
We now analyze the ideals appearing in the intersection in the LHS of~\eqref{eq: rev incl take 2} individually. Fix a  $k$
with $1 \leq k \leq \ell$. We consider cases. Suppose first that $\alpha_k \leq j$. In this case 
$\left[\langle y_0,y_1^{\alpha_k},L_{A_k} \rangle\right]_{i,j} = S_{i,j}$ since 
every monomial of degree $(i,j)$ with $j \geq \alpha_k$ is either divisible by $y_0$
or $y_1^{\alpha_k}$. 
On the other hand, suppose that $j < \alpha_k$. Then we have
$$\left[\langle y_0,y_1^{\alpha_i},L_{A_i} \rangle\right]_{i,j} = 
\left[\langle y_0,L_{A_i} \rangle\right]_{i,j}$$
since no monomial containing $y^{\alpha_k}$ contributes to the $(i,j)$-th graded piece of
the ideal if $\alpha_k > j$. In summary, we have shown that to prove~\eqref{eq: rev incl take 2} it suffices to show 
\begin{equation}\label{vs}
\left[ \bigcap_{
\scriptsize{
\begin{array}{c}
A_k \in \pi_1(X) \\
\alpha_k > j
\end{array}}}
\langle y_0,L_{A_k} \rangle \right]_{i,j} \subseteq \left[\langle I_X,y_0 \rangle \right]_{i,j}.
\end{equation}
Because the vector space on the RHS of ~\eqref{vs} is a subspace
of the vector space on the LHS, it is enough
to show that the two vector spaces in ~\eqref{vs} have the same
dimension.

The ideal on the LHS of~\eqref{vs} is the ideal of the points
$$Y = \{A_k \times [0:1] ~|~ \alpha_k > j \} \subseteq \mathbb{P}^1 \times \mathbb{P}^1.$$
The number of points in this set is the number of $\alpha_k \in \alpha_X$
with $\alpha_k \geq j+1$ which is by definition $\alpha^\star_{j+1}$.  Moreover,
by Lemma \ref{pointsonruling} (and Remark \ref{swapruling}), we have
$H_Y(i,j) = \alpha^\star_{j+1}$ since $i \geq a \geq |\pi(X)|-1 = \alpha_1^{\star}-1$ by assumption
and $\alpha_1^\star \geq \alpha_{j+1}^\star$ by properties of partitions.  On the other hand,
because $y_0$ is a non-zero-divisor 
and since $i \geq |\pi_1(X)|-1$, by Corollary \ref{specialvaluesofHX2} we also have
$$H_{S/\langle I_X,y_0 \rangle}(i,j) = \alpha_{j+1}^\star.$$
Since $H_Y(i,j) = \dim_{\Bbbk} S_{i,j} - \dim_{\Bbbk} [I_Y]_{i,j}$ and 
$H_{S/\langle I_X,y_0 \rangle}(i,j) =\dim_{\Bbbk} S_{i,j} 
-\dim_{\Bbbk} [\langle I_X,y_0\rangle]_{i,j}$, we can conclude that the two vector spaces in \eqref{vs} have the same dimension, as desired.

To complete the proof, we show that all the ideals
on the RHS are primary. Recall that a monomial ideal 
$Q = \langle m_1,\ldots,m_t \rangle$ in a polynomial ring 
$\Bbbk[z_1,\ldots,z_n]$ is a primary ideal if and only if, after 
a permutation of the variables, 
$Q = \langle z_1^{a_1},\ldots,z_r^{a_r},m'_1,\ldots,m'_p\rangle$
and the only variables that divide $m'_1,\ldots,m'_p$
are $z_1,\ldots,z_r$ (e.g, see
\cite[Proposition 6.1.7]{V}).

It follows that the ideal
$\langle \langle x_0,x_1 \rangle^a,y_0 \rangle$ is 
a primary monomial ideal, since the generators
are $x_0^a,x_1^a,y_0$ and monomials of the
form $x_0^bx_1^c$ with $b+c = a$ and $b,c >0$.  For
the ideals $\langle y_0,y_1^{\alpha_k},L_{A_k}=a_{k,1} x_0 - a_{k,0}x_1 \rangle$,
if $a_{k,0} =0$, then this is also a primary monomial ideal.
If $a_{k,0} \neq 0$, then 
$$S/\langle y_0,y_1^{\alpha_k},a_{k,1}x_0-a_{k,0}x_1 \rangle
= S/\langle y_0,y_1^{\alpha_k},\frac{a_{k,1}}{a_{k,0}}x_0-x_1 \rangle \cong \Bbbk[x_0,y_0,y_1]/\langle y_0,y_1^{\alpha_k},x_0  \rangle.$$
Since $\langle y_0,y_1^{\alpha_k},x_0  \rangle$ is 
primary, so is  $\langle y_0,y_1^{\alpha_k},L_{A_i} \rangle$.
\end{proof}

The {\it virtual dimension} of a
graded $S$-module $M$, denoted 
${\rm vdim}(M)$, is the minimum length
of a virtual resolution of $M$.  By
\cite[Proposition 2.5]{Virtual}, we always
have ${\rm vdim}(M) \geq {\rm codim}(M)$.  
We say $M$ is {\it virtually Cohen-Macaulay}
if ${\rm vdim}(M) = {\rm codim}(M)$.
Our main result implies the following corollary.

\begin{corollary}
Let $X$ be a set of points in $\mathbb{P}^1 \times 
\mathbb{P}^1$.  Then $S/I_X$ is virtually
Cohen-Macaulay.
\end{corollary}

\begin{proof}
For any finite set of points
$X$ in $\mathbb{P}^1 \times \mathbb{P}^1$,
we have ${\rm codim}(S/I_X) = 2$.  The result now follows from Theorem \ref{maintheorem2}
which shows ${\rm vdim}(S/I_X) \leq 2$.
\end{proof}

\begin{remark}
The above fact was already observed in
\cite{Virtual} (in particular, see the comments
after \cite[Theorem 4.1]{Virtual}), although the terminology
of  virtually Cohen-Macaulay was not used.  It was also observed in \cite[Corollary 4.2]{Virtual} that the ring $S/(I_X \cap \langle x_0,x_1\rangle^a)$ is
Cohen-Macaulay for $a \gg 0$.  This result
can also be recovered from Theorem \ref{maintheorem2};  in fact, our proof shows
that one can take $a \geq |\pi_1(X)|-1$.
The virtually Cohen-Macaulay property is further explored in \cite{Berkesch}.
\end{remark}

We conclude with some examples which illustrate our results.   Our 
first example shows that the uniform bound of Theorem \ref{maintheorem2} is
optimal.

\begin{example}
Let $X$ be a set of six points with generic Hilbert function, as in 
Example \ref{runningexample}.  For this example, $\lvert \pi_1(X) \rvert -1 = 5$,
and thus, 
by Theorem \ref{maintheorem2},  the minimal graded free
resolution of $S/(I_X \cap \langle x_0,x_1 \rangle^a)$ has length two for
all $a \geq 5$.  A direct computation using a computer algebra system 
shows that $S/(I_X \cap \langle x_0,x_1 \rangle^a)$ has a minimal
graded free resolution of length three for
$0 \leq a < 5$.  For example, when $a=4$, 
Macaulay2 yields the minimal free resolution:
$$0 \rightarrow 
S \rightarrow S^6 \rightarrow S^6 \rightarrow S \rightarrow S/(I_X\cap\langle x_0,x_1\rangle^4) \rightarrow 0,$$
where we have dropped the bigraded shifts for readability.
Thus, the uniform bound for all points in 
$\mathbb{P}^1 \times \mathbb{P}^1$ of Theorem \ref{maintheorem2} cannot be improved.
\end{example}

For a specific set of points $X$ in $\mathbb{P}^1 \times 
\mathbb{P}^1$, however, it is possible that the smallest integer $a$ such that
$S/(I_X \cap \langle x_0,x_1 \rangle^a)$ has a minimal resolution of length two
may be strictly smaller than $|\pi_1(X)|-1$.  For example, if  $X$ has the property that $S/I_X$ is a Cohen-Macaulay ring (such $X$ are classified in \cite{FGM,GMR2,G,Adam}), then
$S/(I_X \cap \langle x_0,x_1 \rangle^a)$ has a minimal graded free resolution
of length two for all $a \geq 0$.  Below, we give another example of
this phenomenon.

\begin{example}\label{nonboundexample}
Let $P_1,P_2,P_3,P_4$ be four distinct points in $\mathbb{P}^1$ and let
$Q_1,Q_2,Q_3,Q_4$ be another collection of four distinct points in $\mathbb{P}^1$ (we
allow for the case that $P_i = Q_j$ for some $i$ and $j$).
Consider the following set of nine points in $\mathbb{P}^1 \times \mathbb{P}^1$:
$$X = \{P_1 \times Q_1, P_1 \times Q_2, P_1 \times Q_3, P_1\times Q_4,
P_2 \times Q_2, P_2 \times Q_4, P_3 \times Q_3, P_3 \times Q_4, P_4 \times Q_4\}.$$
We can visualize these points as follows:
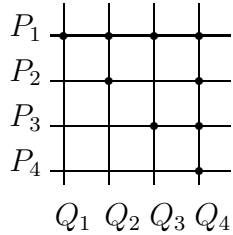
\begin{center}
\setlength{\unitlength}{0.6mm}
\begin{picture}(44,60)(-13,-15)
\put(-2,1){\line(1,0){40}}
\put(-2,11){\line(1,0){40}}
\put(-2,21){\line(1,0){40}}
\put(-2,31){\line(1,0){40}}
\put(1,-2){\line(0,1){40}}
\put(11,-2){\line(0,1){40}}
\put(21,-2){\line(0,1){40}}
\put(31,-2){\line(0,1){40}}
\put(1,31){\circle*{2}}
\put(11,31){\circle*{2}}
\put(21,31){\circle*{2}}
\put(31,31){\circle*{2}}
\put(31,21){\circle*{2}}
\put(31,11){\circle*{2}}
\put(11,21){\circle*{2}}
\put(21,11){\circle*{2}}
\put(31,1){\circle*{2}}
\put(-11,31){$P_1$}
\put(-11,21){$P_2$}
\put(-11,11){$P_3$}
\put(-11,1){$P_4$}
\put(-1,-11){$Q_1$}
\put(10,-11){$Q_2$}
\put(20,-11){$Q_3$}
\put(30,-11){$Q_4$}
\end{picture}
\end{center}
Consequently, $|\pi(X)|-1 =4-1=3$.

On the other hand,  using Macaulay2, we find that the smallest $a$ for which 
$S/(I_X \cap \langle x_0,x_1 \rangle^a)$ has minimal resolution of 
length two is $a=2 < |\pi_1(X)|-1$.  
\end{example}

We round out this paper with some observations and questions.  

\begin{remark} It is
well-known that for a set of points in $Z$ in $\mathbb{P}^1$, the
Castelnuovo-Mumford regularity of $R/I_Z$  (where $I_Z$ is the defining
ideal of $Z$ in $R = \Bbbk[x_0,x_1]$) is given by 
${\rm reg}(R/I_Z) = |Z|-1$.  If we combine this fact with Theorem
\ref{maintheorem2}, we can rewrite the bound in terms of the regularity
of the set of points in the projection, i.e., $a \geq {\rm reg}(R/I_Z)$ where
$Z = \pi_1(X) \subseteq \mathbb{P}^1$.  This suggests that one 
way to generalize Theorem \ref{maintheorem2} to $\mathbb{P}^n \times
\mathbb{P}^m$ (or even $\mathbb{P}^{n_1} \times \cdots \times \mathbb{P}^{n_r}$)
is to consider the regularity of $R/I_Z$ where $Z = \pi_1(X) \subseteq 
\mathbb{P}^n$ (or the multigraded regularity in the case of more than two
projective spaces).
\end{remark}

\begin{remark}
Example \ref{nonboundexample} shows that for some specific configurations of points
in $\mathbb{P}^1 \times \mathbb{P}^1$, the bound in Theorem \ref{maintheorem2} can be improved. 
For example, if we know that the set of 
points has the property that $S/I_X$ is a Cohen-Macaulay ring, then 
we can take $a=0$.  Moreover, these points can be classified geometrically.  The 
points in Example \ref{nonboundexample} do not have the Cohen-Macaulay property, but
they are not as ``sparse'' as a set of points with generic Hilbert function.    It would be interesting to determine what conditions on the geometry of the points allow
us to improve our bounds on the existence of short virtual resolutions.
\end{remark}

\end{document}